\documentclass[leqno,12pt]{article} 

\usepackage{amsfonts, amsmath,amssymb, amsthm}
\usepackage{verbatim}
\usepackage[all,2cell]{xy}
\objectmargin+{1mm}

\theoremstyle{plain}  
\newtheorem{thm}{Theorem}[section]
\newtheorem{cor}[thm]{Corollary}
\newtheorem{lem}[thm]{Lemma}
\newtheorem{prop}[thm]{Proposition}

\theoremstyle{definition}
\newtheorem{df}[thm]{Definition}

\theoremstyle{remark}
\newtheorem*{claim}{Claim}

\DeclareMathOperator{\id}{id}
\DeclareMathOperator{\rinf}{\rightarrowtail}
\DeclareMathOperator{\rdef}{\twoheadrightarrow}

\renewcommand{\Im}{\operatorname{Im}}

\newcommand{\Cyl}{\operatorname{Cyl}}
\newcommand{\Cone}{\operatorname{Cone}}

\newcommand{\Ar}{\operatorname{Ar}}
\newcommand{\Ob}{\operatorname{Ob}}
\newcommand{\dom}{\operatorname{dom}}
\newcommand{\ran}{\operatorname{ran}}
\newcommand{\op}{\operatorname{op}}

\newcommand{\Mor}{\operatorname{Mor}}

\newcommand{\qis}{\mathrm{qis}}

\newcommand{\cf}{\textrm{cf.}\;}
\newcommand{\ul}[1]{\underline{#1}}
\newcommand{\onto}[1]{\stackrel{#1}{\to}}
\newcommand{\onot}[1]{\stackrel{#1}{\leftarrow}}
\newcommand{\epi}{\twoheadrightarrow}
\newcommand{\mono}{\rightarrowtail}

\newcommand{\Ac}{\operatorname{Ac}}

\newcommand{\bK}{\operatorname{\mathbb{K}}}
\newcommand{\bE}{\mathbf{E}}

\newcommand{\cC}{\mathcal{C}}
\newcommand{\cD}{\mathcal{D}}
\newcommand{\cE}{\mathcal{E}}

\newcommand{\cS}{\mathcal{S}}

\newcommand{\Ch}{\operatorname{Ch}}
\newcommand{\Chb}{\Ch^b}

\newcommand{\ChbE}{\Chb(\cE)}
\newcommand{\ChbZ}{\Chb(\Z)}
\newcommand{\ChE}{\Ch(\cE)}
\newcommand{\Chsh}{\Ch^{\#}}
\newcommand{\ChshE}{\Chsh(\cE)}

\newcommand{\cN}{\mathcal{N}}

\newcommand{\cT}{\mathcal{T}}

\newcommand{\DE}{\cD(\cE)}
\newcommand{\E}{\cE}

\newcommand{\Ebar}{\underline{\cE}}
\newcommand{\Ewbar}{\underline{\cE}^w}

\newcommand{\NCE}{\operatorname{\bf NC}(\cE)}
\newcommand{\NC}[1]{\mathrm{({\bf{NC\ #1}})}}

\newcommand{\Nbar}{\underline{\cN}}

\newcommand{\TE}{\cT(\bE)}

\newcommand{\Z}{\mathbb{Z}}
\newcommand{\WEE}{\operatorname{\bf WE}(\cE)}

\renewcommand{\coprod}{\sqcup}
\newcommand{\isomto}{\onto{\simeq}}
\newcommand{\WE}[1]{\mathrm{({\bf{WE\ #1}})}}
\newcommand{\WEp}[1]{\mathrm{({\bf{WE\ #1}})'}}
\newcommand{\WEop}[1]{\mathrm{({\bf{WE\ #1}})^{\op}}}
\newcommand{\WEpop}[1]{\mathrm{{({\bf{WE\ #1}})'}^{\op}}}

\def\sn{\smallskip\noindent}

\def\bn{\bigskip\noindent}

\title{Quasi-weak equivalences in complicial exact categories}

\author{Toshiro Hiranouchi and Satoshi Mochizuki}

\begin{document}
\maketitle
\begin{abstract}
We introduce a notion of quasi-weak equivalences 
associated with weak-equivalences 
in an exact category. 
It gives us a delooping for (idempotent complete) exact categories 
and a condition that  
the negative $K$-group of an exact category becomes trivial. 
\end{abstract}

\section{Introduction}
The negative $K$-theory $\bK(\cE)$ for an exact category $\cE$ 
is introduced in \cite{Sch04} and \cite{Sch06} 
by M.\ Schlichting. 
This generalizes 
the definition of Bass, Karoubi, Pedersen-Weibel, Thomason, Carter and Yao. 
The first motivation of our work 
is to investigate some vanishing conjectures 
of such negative $K$-groups: 

\sn
(a) For any noetherian scheme $X$ of Krull dimension $d$, 
$K_{-n}(X)$ is trivial for $n>d$ (\cite{Wei80}). 

\sn 
(b) 
The negative $K$-groups of a small abelian category is trivial (\cite{Sch06}). 

\sn 
(c) 
For a finitely presented group $G$, $K_{-n}(\Z G) = 0$ for $n>1$ 
(\cite{Hsi84}). 

In \cite{Sch06}, 
it was given 
a description of $\bK_{-1}(\cE)$ 
and 
a condition on vanishing of $\bK_{-1}(\cE)$ 
for an (essentially small) exact category $\cE$ 
in terms of its {\it unbounded} derived category $\cD(\cE)$: 
We have 
$\bK_{-1}(\cE)=\bK_0(\cD(\cE))$ 
and $\bK_{-1}(\cE)$ is trivial 
if and only if $\cD(\cE)$ is idempotent complete 
(= Karoubian in the sense of \cite{TT90}, A.6.1). 
To extend this observation, we shall introduce 
the notion of {\it{higher derived categories}} $\cD_n(\cE)$ 
and 
show the following theorem:

\begin{thm}[Cor.\ \ref{cor:obst}]
For an idempotent complete exact category $\cE$, 
we have 
$\bK_{-n}(\cE)= \bK_0(\cD_n(\cE))$. 
Moreover, $\bK_{-n}(\cE)$ is trivial if and only if 
$\cD_n(\cE)$ is idempotent complete. 
\end{thm}

Although we limit our consideration to idempotent complete 
exact categories 
to avoid some technical difficulties, 
the exact categories in the conjectures (a)-(c) above 
satisfies this condition. 
Recall that the derived category $\DE$ 
is the triangulated category obtained by 
formally inverting quasi-isomorphisms in 
the category of chain complexes $\ChE$. 
The pair $(\ChE, \qis)$ 
of the category of chain complexes $\ChE$ 
and $\qis$ the class of quasi-isomorphisms 
forms a {\it complicial} exact category with weak equivalences 
(\cf Def.\ \ref{df:complicial exact category}). 
More generally, 
for a complicial exact category  
with weak equivalences
$\bE=(\cE,w)$ (\cf Def.\ \ref{df:complicial exact category with we}), 
we define a class of weak equivalences $qw$ 
in the category of chain complexes 
$\ChE$,  
which is called {\it{quasi-weak equivalences}} 
associated with $w$. 
If $w$ is the class of isomorphisms in $\cE$, 
then $qw$ is just the class of quasi-isomorphisms on $\ChE$. 
The derived category $\cD(\bE)$ of $\bE$ is obtained by 
formally inverting the quasi-weak equivalences 
in $\Ch(\cE)$. 
Put $\Ch(\bE)=(\ChE,qw)$ and 
one can define the class weak equivalences 
in $\Ch_n(\bE) := \Ch(\Ch_{n-1}(\bE))$ inductively. 
The $n$-th derived category $\cD_n(\bE)$ associated with $\bE$,  
is the derived category of $\Ch_n(\bE)$. 
We also obtain the following theorems 
on the negative $K$-theory $\bK(\bE)$ for $\bE$ (for definition, see \cite{Sch09}): 

\begin{thm}[Thm.\ \ref{partial GW}]
Assume that $\cE$ is idempotent complete. 
Then we have:

\sn
$\mathrm{(i)}$ {\rm (Gillet-Waldhausen theorem)} 
$\bK(\bE)\isomto\bK(\Ch^b(\bE))$, 

\sn
$\mathrm{(ii)}$ {\rm (Eilenberg swindle)} 
$\bK(\Ch^{+}(\bE))\isomto\bK(\Ch^{-}(\bE))\isomto 0$, 

\sn
$\mathrm{(iii)}$ {\rm (Delooping)} 
$\bK(\Ch(\bE))\isomto\Sigma\bK(\bE)$, 
where $\Sigma$ is a suspension functor 
on the stable category of spectra.
\end{thm}

The organization of this note is as follows: 
In Section\ \ref{sec:we}, we list 
several axioms about weak equivalences 
in a category with cofibrations 
and study their implication. 
In Section \ref{sec:complicial}, 
after recalling the definition of 
complicial exact category with weak equivalences, 
we consider the notion of {\it null classes} 
and investigate the relation to weak equivalences 
in a complicial exact category. 
In Section\ \ref{sec:deloop}, 
we introduce the quasi-weak equivalences 
as noted above 
which is a class of weak equivalences in 
the exact category of chain complexes $\ChE$   
associated with a given weak equivalences $w$  
in an exact category $\cE$.
By using this, 
we prove the main theorem. 
%
Throughout this note, 
we follow basically 
the terminologies on algebraic $K$-theory in \cite{TT90} and \cite{Sch09}.

\medskip\noindent
{\it Acknowledgments.} 
The second author is grateful 
for Marco Schlichting for stimulating discussion 
about a complicial Gabriel-Quillen embedding.  
The first author has been partially supported by 
JSPS KAKENHI \#21740015. 
This work has also been supported by
JSPS Core-to-Core Program \#18005 (Coordinator: Makoto Matsumoto).
The final part of this note was written during 
a stay of the first author at the Duisburg-Essen university. 
He wishes to thank the institute for its hospitality.

\section{Weak equivalences in categories with cofibrations}
\label{sec:we}
In this section, 
we list several axioms on weak equivalences 
in a category with cofibrations 
and study their implications.   
Let $\cC$ be a category with cofibrations 
and $w$ be a class of morphisms in $\cC$.  
We denote by $\Ar(\cC)$ the category of arrows $\cC \to \cC$. 
The functors $\dom, \ran:\Ar(\cC) \to \cC$ are defined by 
$\dom(f) = x$  and $\ran(f) = y$ respectively, for $f:x\to y \in \Ar(\cC)$.

First, we consider the following axioms on $w$: 

\sn
$\WE{1}$ Every isomorphisms in $\cC$ is in $w$.

\sn
$\WE{2}$ 
For composable morphisms $f$ and $g$ in $\cC$, 
if two of $f$, $g$ and $g f$ are in $w$, 
then the other one is also in $w$.

\sn
$\WE{3}$ 
For a commutative diagram in $\cC$,
\begin{equation}
{\label{diagram WE3}}
\vcenter{
\xymatrix{
x \ar@{>->}[r] \ar[d]_{a} & y \ar@{->>}[r] \ar[d]_{b} & z \ar[d]_{c}\\ 
x' \ar@{>->}[r] & y' \ar@{->>}[r] & z'& ,
}
}
\end{equation}
where the horizontal lines are cofibration sequences, 
if $a$ and $c$ are in $w$, 
then so is $b$.

\sn
$\WEp{3}$ For the commutative diagram (\ref{diagram WE3}) in $\cC$,
if $a$ and $b$ are in $w$, 
then so is $c$.

\sn
$\WE{4}$ 
For a commutative diagram in $\cC$, 
\begin{equation}
{\label{import cobase diagram}}
\vcenter{
\xymatrix{
y \ar[d]_{a} & x \ar[r]^{f} \ar@{>->}[l]_{i} \ar[d]_{b} & z \ar[d]_{c}\\ 
y' &x' \ar[r]_{f'} \ar@{>->}[l]^{i'} & z' &,
}}
\end{equation}
where $i$ and $i'$ are cofibrations, 
if $a$, $b$ and $c$ are in $w$, 
then the induced map  
$a \coprod_{b} c:y \coprod_x z \to y' \coprod_{x'} z'$ 
on pushouts is also cofibration.

\sn
$\WE{5}$ 
For any cofibration $x\mono y$ in  $\cC$  
and $x\to z$ in $w$, 
the induced morphism $y \to y\coprod_x z$ is in $w$.

\sn
$\WE6$ (Factorization axiom) There are 
a functor 
$\Cyl:\Ar\cC \to \cC$ 
and, 
natural transformations $\alpha:\dom \Rightarrow \Cyl$ 
and $\beta:\Cyl \Rightarrow \ran$ 
such that for any morphism $f:x \to y$ in $\cC$, 
$\alpha(f):x \to \Cyl(f)$ is a cofibration, 
$\beta(f):\Cyl(f) \to y$ is in $w$ and
$\beta(f)\alpha(f)=f$. 

\sn
$\WE7$ 
For a commutative diagram 
\begin{equation}
{\label{diagram WE7}}
\vcenter{
\xymatrix{
&x \ar[d]_{a}\ar[r]^i & y \ar[d]_{b} \\ 
&x'\ar[r]^{i'} & y' &, 
}
}
\end{equation}
with retractions $p:y \to x$ and $p':y\to x'$ such that 
$pi = \id_x$ and $p'i' = \id_{x'}$,  
if $b$ is in $w$, so is $a$.

\bn
Note that the axiom $\WE{2}$ implies that the class $w$  
is closed under composition. 
Next we study logical relations among these axioms 
as follows:

\begin{prop}[\cf \cite{Gun78}; \cite{GJ99}, 8.8]
{\label{we implications}}
$\mathrm{(i)}$ $\WE{1}$ and $\WE{3}$ imply $\WE{5}$.

\sn
$\mathrm{(ii)}$ $\WE{2}$, $\WE{5}$ and $\WE{6}$ 
imply $\WE{4}$. 

\sn
$\mathrm{(iii)}$ $\WE1$ and $\WE{4}$ imply $\WEp{3}$.
\end{prop}

\begin{proof}
$\mathrm{(i)}$ 
Let $i:x \mono y$ be a cofibration in $\cC$ 
and $a:x \to x'$ in $w$. 
It is easy to see that the sequence 
$x' \mono y':=x'\coprod_xy \epi x/y$ 
is a cofibration sequence. 
Now applying $\WE{3}$ to the following commutative diagram
$$
\xymatrix{
x \ar@{>->}[r]^{i} \ar[d]_a \ar@{}[rd] & y \ar@{->>}[r] \ar[d]_b & x/y \ar[d]_{\id_{x/y}}\\
x' \ar@{>->}[r] & y' \ar@{->>}[r] & x/y& ,
}
$$
we learn that $b:y \to y'$ is also in $w$.

\sn
$\mathrm{(iii)}$ 
Now we consider the following commutative diagram in $\cC$:
$$
\xymatrix{
x \ar@{>->}[r] \ar[d]_{a} & y \ar@{->>}[r] \ar[d]_{b} & z \ar[d]_{c}\\ 
x' \ar@{>->}[r] & y' \ar@{->>}[r] & z'& ,
}
$$
where the horizontal lines are cofibration sequences, 
and $a$ and $b$ are in $w$. 
Note that the map $0 \to 0$ is in $w$ by $\WE{1}$. 
For the commutative diagram (\ref{diagram WE3}) in $\cC$,
we assume that $a$ and $b$ are in $w$. 
The following diagrams are coCartesian:
$$
\xymatrix{
&x \ar[d]\ar@{>->}[r]\ar[r] & y \ar@{->>}[d]\\ 
&0\ar@{>->}[r] \ar[r] & z & ,
}
\quad
\xymatrix{
&x' \ar[d]\ar@{>->}[r]\ar[r] & y' \ar@{->>}[d]\\ 
&0\ar@{>->}[r] \ar[r] & z' &.
}
$$ 
By $\WE4$, the map $c = 0 \coprod_{a}b$  is in $w$. 

\sn
$\mathrm{(ii)}$ 
Consider the commutative diagram (\ref{import cobase diagram}) in $\WE4$.  
First we assume the lemma below 
we intend to conclude the result.

\begin{lem}
{\label{precogluing axiom}}
Let us assume the axioms $\WE2$ and $\WE5$. 
In the diagram $\mathrm{(\ref{import cobase diagram})}$, 
suppose that both $f$ and $f'$ are cofibrations or in $w$. 
Then, $a \coprod_b c: y \coprod_x z \to y' \coprod_{x'} z'$ is in $w$. 
\end{lem}

Applying $\WE6$ to the diagram (\ref{import cobase diagram}), 
we get the following diagrams

$$
\xymatrix{
y \ar[d]_{a} & x \ar@{>->}[r]^-{\alpha(f)} \ar@{>->}[l]_{i} \ar[d]_{b} & 
\Cyl(f) \ar[r]^-{\beta(f)} \ar[d]_d & z \ar[d]_{c}\\ 
y' & x' \ar@{>->}[r]_-{\alpha(f')} \ar@{>->}[l]^{i'} & 
\Cyl(f') \ar[r]_-{\beta(f')} & z'& ,
}
$$
$$
\xymatrix{
y \coprod_x \Cyl(f) \ar[d]_{e} & 
\Cyl(f) \ar[r]^{\ \ \beta(f)} \ar[d]_d \ar@{>->}[l] & z \ar[d]_{c}\\ 
y' \coprod_{x'} \Cyl(f') &  \Cyl(f') \ar[r]_{\ \ \beta(f')} \ar@{>->}[l] & z'& .}
$$
By Lemma\ \ref{precogluing axiom}, 
we learn that $e$ is in $w$.
Again using Lemma\ \ref{precogluing axiom}, 
we learn that $e \coprod_d c$ is in $w$. 
Thus by Lemma \ref{coCartesian lemma}, (i) below, 
we notice that $e \coprod_d c$ is $a \coprod_b c$. 
\end{proof}

\begin{lem}
{\label{coCartesian lemma}}
Let us consider a commutative diagram below in a category.
$$
\xymatrix{
\bullet \ar[r] \ar[d] \ar@{}[rd]|{\mathrm{I}} & 
\bullet \ar[r] \ar[d] \ar@{}[rd]|{\mathrm{II}} &
\bullet \ar[d]\\
\bullet \ar[r] & 
\bullet \ar[r] & 
\bullet  & .
}$$

\sn
$\mathrm{(i)}$ If the diagrams $\mathrm{I}$ and 
$\mathrm{II}$ are coCartesian squares, 
then the diagram $\mathrm{I+II}$ is also. 

\sn
$\mathrm{(ii)}$ If the diagram $\mathrm{I+II}$ and 
$\mathrm{I}$ are coCartesian squares, 
then the diagram $\mathrm{II}$ is also.
\end{lem}

\begin{proof}[Proof of Lemma \ref{precogluing axiom}]
First let us assume that 
both $f$ and $f'$ in the diagram (\ref{import cobase diagram}) 
are in $w$. 
Then in the diagram below
$$
\xymatrix{
y \ar[r]^-{g} \ar[d]_a & y \coprod_x z \ar[d]_{a \coprod_b c}\\
y' \ar[r]_-{g'} & y' \coprod_{x'} z' & ,
}$$
$g$ and $g'$ are in $w$ by $\WE5$. 
Therefore by $\WE2$, 
$a \coprod_b c$ is also. 
\sn
Next let us suppose that 
$f$ and $f'$ in the diagram (\ref{import cobase diagram}) 
are cofibrations in $\cC$. 
Consider the following diagram:
$$
\xymatrix{
  & x \ar@{>->}[rr] \ar@{>->}[ld] \ar[dd] && z \ar@{>->}[ld] \ar[d]\\
y \ar[dd] \ar@{>->}[rr] && y \coprod_x z \ar[d] 
  & x' \coprod_x z \ar[d] \ar@{>->}[ld]\\
& x' \ar@{>->}[ld] \ar@{>->}'[r][rr] \ar[rru] 
& y' \coprod_{y} y \coprod_x z  \ar[d] & z' \ar@{>->}[ld]\\
y' \ar@{>->}[rr] \ar@{>->}[rru] && y' \coprod_{x'} z' &.  
}
$$

\begin{claim}
The commutative diagram below is coCartesian.
$$
\xymatrix{
x' \coprod_x z \ar[r] \ar[d] & z' \ar@{>->}[d]\\
y' \coprod_y y \coprod_x z \ar[r] & y' \coprod_{x'} z' &.
}
$$
\end{claim}

\begin{proof}
Consider the following commutative diagram:
$$
\xymatrix{
x' \ar[r] \ar[d] \ar@{}[rd]|{\mathrm{I}} 
& x' \coprod_x z \ar[d] \ar[r] \ar@{}[rd]|{\mathrm{II}}
& z' \ar[d]\\
y' \ar[r] & y' \coprod_y y \coprod_x z \ar[r] & y' \coprod_{x'} z' & .
}
$$
The squares $\mathrm{I+II}$ and $\mathrm{I}$ are coCartesian. 
Therefore by Lemma\ \ref{coCartesian lemma} (ii), 
the diagram $\mathrm{II}$ is also coCartesian.
\end{proof}

By ({\bf{WE 5}}), 
$y \coprod_x z \to y' \coprod_y y \coprod_x z$ and $z \to x' \coprod_x z$ 
are in $w$. 
Therefore by ({\bf{WE 2}}), 
$x' \coprod_x z \to z'$ is in $w$. 
Hence by ({\bf{WE 5}}) again, 
$y' \coprod_y y \coprod_x z \to y' \coprod_{x'} z'$ is in $w$.
Finally by ({\bf{WE 2}}), 
the composition $a \coprod_b c:y \coprod_x z \to y' \coprod_y y \coprod_x z \to y' \coprod_{x'} z'$ is in $w$. 
\end{proof}

In a category with fibrations $\cC$  
which is the dual concept of categories with cofibrations, 
we consider the following axioms:

\sn
$\WEpop{3}$ For the commutative diagram (\ref{diagram WE3}) in $\cC$,
if $b$ and $c$ are in $w$, 
then so is $a$.

\sn
$\WEop{4}$ For a commutative diagram in $\cC$ 
$$
\xymatrix{
y \ar[d]_{a} \ar@{->>}[r]^{p} & x  \ar[d]_{b} & z \ar[d]_{c} \ar[l]\\ 
y' \ar@{->>}[r]_{p'} & x'  & z' \ar[l]& ,
}
$$
where $p$ and $p'$ are fibrations, 
if $a$, $b$ and $c$ are in $w$, 
then $a \times_{b} c:y \times_x z \to y' \times_{x'} z'$ is also.

\sn
$\WEop{5}$ $w$ is stable under base change by fibrations. 
That is, for any fibration $y\epi x$ in  $\cC$  
and $z\to x$ in $w$, 
the induced morphism $y\times_x z \to y$ is in $w$.

\sn
$\WEop{6}$ There are a functor
$M:\Ar\cC \to \cC$ and 
natural transformations $\gamma:\dom \Rightarrow M$ and 
$\delta:M \Rightarrow \ran$ 
such that for any morphism $f:x \to y$ in $\cC$, 
$\gamma(f):x \to M(f)$ is in $w$, 
$\delta(f):M(f) \to y$ is a fibration and
$\delta(f)\gamma(f)=f$.

\bn
Thus, one can establish the dual statement of 
Proposition \ref{we implications}. 
In particular, 
since an exact category $\cE$ is 
an additive category with bifibrations, 
we can apply $\cE$ to Proposition\ \ref{we implications} 
and its dual statement.

\section{Complicial exact category}
\label{sec:complicial}
We recall the theory of complicial exact categories with weak equivalences 
following \cite{Sch09}.
Let $\ChbZ$  be the exact category 
of bounded chain complexes of finitely generated free $\Z$-modules. 
Its exact structure is given by 
the degree-wise split sequences. 
There is a symmetric monoidal tensor product 
$\otimes: \ChbZ \times \ChbZ \to \ChbZ$ 
which extends the usual tensor product of free $\Z$-modules 
defined by 
$(a\otimes b)^n := \bigoplus_{i+j=n} a^i \otimes b^j$. 
Its differential $d^n: (a\otimes b)^n \to (a\otimes b)^{n+1}$ is 
$$
 d^i_a \otimes \id_{b^j} +(-1)^i \id_{a_i} \otimes d^j_b:a^{i}\otimes b^j \to (a^{i+1}\otimes b^j) \oplus  (a^i\otimes b^{j+1})  
$$
on $a^i\otimes b^j \subset (a\otimes b)^n$ ($i+j=n$). 
The unit is the chain complex $1\!\!1$ which is $\Z$  in degree $0$  
and $0$ elsewhere. 
The complex $C$ is $\Z$  in degrees $0$ and $-1$  
and is $0$ otherwise. 
The only non-trivial differential is $d^{-1} = \id_{\Z}$. 
The complex $T$ is $\Z$ in degree $-1$ and $0$ elsewhere. 
Note that we have a conflation of chain complexes 
$1\!\!1 \mono C \epi T$.

\begin{df}[\cite{Sch09}, Def.\ 6.2]
\label{df:complicial exact category}
 An exact category  $\cE$  is said to be {\it complicial}\/ 
if it is equipped with a bi-exact action 
  $\otimes: \ChbZ \times \cE \to \cE$ 
of the symmetric monoidal category $\ChbZ$  on  $\cE$. 
For an object $x\in \cE$, 
put $Cx := C\otimes x$ and $Tx := T\otimes x$. 
A sequence $x \mono Cx \epi Tx$ forms a conflation. 
\end{df}

For any morphism $f :x\to y$ in a complicial exact category $\cE$, 
the {\it cone} $\Cone(f)$ is defined by 
the push-out of $f$ along the inflation $x \mono Cx$. 
We call $\Cyl(f) := y \oplus Cx$ the {\it cylinder} of $f$. 
These make the following conflations:
$y \mono \Cone(f) \epi Tx$, and $x \mono \Cyl(f) \epi \Cone(f)$.
We associate a triangulated category 
$\Ebar$ called the {\it stable category} for  
a complicial exact category $\cE$ as follows:   
An inflation $i: x\rinf y$ in $\cE$ 
is called a {\it Frobenius inflation} 
if for any object $u$ and 
a morphism $f: x \to Cu$ in $\cE$, 
there is a morphism $g:y \to Cu$ such that $f=gi$. 
A deflation $p:x \rdef y$ in $\cE$ 
is called a {\it Frobenius deflation} 
if for any object $u$ and 
a morphism $f:Cu \to y$ in $\cE$, 
there is a morphism $g:Cu \to Cx$ such that $f=pg$. 
The category $\cE$ endowed with 
the class of Frobenius inflations and Frobenius deflations is 
a Frobenius exact category. 
That is,  
it has enough projective and injective objects 
and the class of projective objects and 
the injective objects coincide.
An object $x$ is a projective-injective object 
if and only if it is a direct summand of 
$Cu$ for some object $u$ in $\cE$ (\cite{Sch09}, Lem.\ B.16). 
For any morphism $f:x\to y$ in $\cE$, 
$x \mono Cx$ is a Frobenius inflation, 
thus the conflations 
$x \mono Cx \epi Tx$ and $x \mono \Cyl(f) \epi \Cone(f)$ 
are Frobenius conflations. 
From now on, we always consider the complicial exact category 
$\cE$ as the Frobenius category. 
Recall that 
two maps $f,g :x\to y$ in $\cE$ are called {\it homotopic}, 
if their difference factors through a projective-injective object. 

\begin{lem}
\label{lem:cc}
For any object $x$ in $\cE$,  
$Cx = C\otimes x$ is contractible.
\end{lem}
\begin{proof}
In $\ChbZ$, the identity map $C \to C$ 
factor through $CC$. 
Hence $Cx$ is contractible.
\end{proof}

The {\it stable category} $\Ebar$ of the Frobenius category 
$\cE$ 
is the category whose objects are the objects of $\cE$  
and whose morphisms are the homotopy classes of maps in $\cE$. 
It is known that the stable category $\Ebar$  
is a triangulated category. 
Distinguished triangles in $\Ebar$  
are those triangles which are isomorphic in $\Ebar$ to sequences 
of the form 
\begin{equation}
\label{eq:dt}
x \onto{f} y \to  \Cone(f) \to Tx
\end{equation}
for $f:x\to y$ in $\cE$.

\begin{df}[\cite{Sch09}, Def.\ 6.9]
\label{df:complicial exact category with we}
A class of morphisms $w$ in a complicial exact category $\cE$ 
is called 
a {\it{class of weak equivalences}}\/ 
if $w$ satisfies the following conditions: 

\sn
$\WE0$ The tensor product preserves weak equivalences in both variables, 
that is, if $f$  is a homotopy equivalence in $\ChbZ$  and $g$  is in $w$, 
then $f\otimes g$ is in $w$,

\sn
$\WEp1$ Every homotopy equivalence is in $w$,

\sn
$\WE2$, $\WE3$ and $\WE7$ in the last section, 
namely, $w$ satisfies the 2 out of 3 and 
and is closed under extensions and retracts.

We denote the class of all classes of weak equivalences 
by $\WEE$. 
For a class of weak equivalences $w$  in $\cE$, 
we say that the pair $\bE := (\cE,w)$  
is a {\it complicial exact category with weak equivalences}.
\end{df}


Every class of weak equivalences $w$ 
in a complicial exact category $\cE$
satisfies ({\bf{WE 6}}). 
If fact, 
for a morphism $f: x\to y$ in  $\cE$,  
we have a conflation $x \mono \Cyl(f) \epi \Cone(f)$.
By Lemma \ref{lem:cc}, 
$\Cyl(f) = Cx \oplus y$ and $y$ are homotopy equivalence. 
Hence $\WEp1$ implies $\WE6$.
Proposition\ \ref{we implications} says that 
$w$ satisfies ({\bf{WE 4}}) and ({\bf{WE 5}}).\ 
Furthermore, it is easy to verify that 
$w$ satisfies  ({\bf{WE 4}})${}^{\op}$ - ({\bf{WE 6}})${}^{\op}$. 
An object $x$ in $\cE$ is said to be {\it{$w$-trivial}} 
if the canonical map $0 \to x$ is in $w$. 
Note that by the axioms $\WEp1$ and $\WE2$, 
this condition is equivalent to the condition that 
the canonical map $x \to 0$ is in $w$. 
We denote the class of $w$-trivial objects by $\cE^w$ 
and sometimes consider $\cE^w$ as the full subcategory of $\cE$ 
of the $w$-trivial objects.

\begin{lem}
{\label{char of we}}
A morphism $f:x \to y$  in $\cE$ is in $w$ if and only if 
its mapping cone $\Cone(f)$ is in $\cE^w$.
\end{lem}
\begin{proof}
Let us consider the following diagram:
$$\xymatrix{
&x \ar@{>->}[r]^{\alpha(f)} \ar[d]_f & \Cyl(f) \ar[d]^{\beta(f)} \ar@{->>}[r] & \Cone(f) \ar[d]\\
& y \ar[r]_{\id_y} & y \ar[r] & 0 & ,   
}$$
where $\beta(f)$ is in $w$ by ({\bf{WE 6}}).
Assume that $f$ is in $w$. 
Then applying ({\bf{WE 4}}) to the diagram above, 
we learn that $\Cone(f)$  is  $w$-trivial. 
Next suppose that $\Cone(f)$ is $w$-trivial. 
Then applying ({\bf WE 4})$^{\op}$ to the diagram above, 
we learn that $f$ is in $w$.
\end{proof}

Next we introduce a {\it null class}\/ associated with a class of 
weak equivalences in a complicial exact categories 
and establish a bijective correspondence 
between null classes and classes of weak equivalences in 
the complicial exact category. 
This is an analogue of a bijective correspondence 
between thick subcategories and localizing systems in a triangulated category.  

\begin{df}
{\label{Null class df}}
An additive full subcategory $\cN$ of $\cE$  
is called a {\it{null class}} 
if it satisfies the following axioms:

\sn
$\NC1$ 
If $x$ is an object in $\cN$ and $y$ is an object in $\cE$ 
which is homotopy equivalent to $x$, 
then $y$ is also in $\cN$.

\sn
$\NC2$ 
For $a\in \ChbZ$ and $x\in \cN$, 
we have $a\otimes x \in \cN$.

\sn 
$\NC3$ 
For any conflation $x \mono y \epi z$
in $\cE$, 
if $x, z$ are in $\cN$, so is $y$. 
 
\sn
$\NC7$ 
If there are maps $i:x \to y$ and $p:y\to x$ 
with $p i = \id_x$ and $y \in \cN$, 
then $x\in \cN$. 

\sn
We denote the class of all null classes on $\cE$ by $\NCE$. 
\end{df}

For any null class $\cN$ on $\cE$, 
by $\NC3$, 
it becomes an exact category in the natural way 
and by $\NC2$, it is complicial. 
As in the last section, we can consider the stable category $\Nbar$ of $\cN$ 
which is a full triangulated subcategory of $\underline{\cE}$

\begin{lem}
{\label{fund bij}}
Let  $\cE$  be a complicial exact category. 

\sn
$\mathrm{(i)}$ For any class of weak equivalences $w$ on $\cE$, 
the class of $w$-trivial objects $\cE^w$ is a null class.

\sn
$\mathrm{(ii)}$ Conversely, for any null class $\cN$ on $\cE$, 
let us define the class of morphisms $w_{\cN}$ by 
$$
  w_{\cN}:=\{f\in\Mor\cE\ |\ \Cone(f) \in \cN\}.
$$
Then $w_{\cN}$ is a class of weak equivalences. 

\sn
$\mathrm{(iii)}$ 
For the associations $w \mapsto \cE^w$ and 
$\cN \mapsto w_{\cN}$ gives a bijective correspondence between $\WEE$ and $\NCE$.
\end{lem}

\begin{proof}
(i) Since $0$ is $w$-trivial, it is in $\cE^w$. 
In particular $\cE^w$ is non-empty.

\sn
$\NC1$ Let $x$ be a $w$-trivial object in $\cE$ 
and $f:x \to y$ is a homotopy equivalent. 
Then by $\WEp1$, $f$ is in $w$. 
Therefore by $\WE2$, $y \to 0$ is in $w$.

\sn
$\NC2$ Let $x$ be a $w$-trivial object in $\cE$
and $a$ an object in $\ChbZ$. 
Since the end-functor $a \otimes ?$ on $\ChbZ$ is additive, 
there is a canonical isomorphism $0 \isomto a\otimes 0$. 
Since $\otimes$ preserves $w$ $\WE0$, 
the canonical morphism $a\otimes x \to a\otimes 0$ is in $w$.
Therefore by $\WE2$, we learn that $a\otimes x \to 0$ is in $w$.

\sn
$\NC3$ In the diagram below
$$
\xymatrix{
x \ar@{>->}[r] \ar[d] & y \ar[d] \ar@{->>}[r] & z \ar[d]\\
0 \ar@{>->}[r] & 0 \ar@{->>}[r] & 0 &, 
}$$ 
assume that $x$ and $z$ are in $\cE^w$. 
Then by $\WE3$, $y$ is also in $\cE^w$.

\sn
$\NC7$ 
Assume that there are maps $i:x \to y$ and $p:y\to x$ 
with $p i = \id_x$ and $y \in \cE^w$. 
By $\WE7$, we have $x\in\cE^w$. 

\sn
(ii) $\WE0$ 
Let $f:a\to b$ be a homotopy equivalence in $\ChbZ$ 
and $g:x\to y$ in $w_{\cN}$. 
Note that $f\otimes \id_y$ 
is a  homotopy equivalence in $\cE$. 
By $\NC1$, we have $\Cone(f\otimes \id_y)$ in $\cN$. 
From the isomorphism 
$\Cone (\id_a \otimes g)\isomto a \otimes \Cone(g)$ 
and $\NC2$, we have $\Cone(\id_a \otimes g) \in \cN$. 
From the equality 
$f\otimes g = (\id_a \otimes g)  (f\otimes \id_y)$, 
there is a distinguished triangle in $\Ebar$
$$
  \Cone(\id_a \otimes g) \to \Cone(f\otimes g) \to \Cone (f\otimes \id_y) \onto{+1}
$$
by the octahedral axiom. 
Since $\Nbar$ is triangulated, 
$\Cone(f\otimes g)$ in $\Nbar$. 
Then by $\NC1$, it is in $\Nbar$.

\sn 
$\WEp1$ 
Let $f:x \to y$ be a homotopy equivalence in $\cE$. 
Then $\Cone(f)$ is contractible 
(that is, $\Cone(f) \to 0$ is a homotopy equivalence). 
Therefore by $\NC1$, $\Cone(f)$ is in $\cN$.

\sn 
$\WE2$ Let us consider morphisms $x \onto{f} y \onto{g} z$ in $\cE$. 
Then by the octahedral axiom, 
we have a distinguished triangle in $\Ebar$ 
$$
\Cone(f) \to \Cone(gf) \to \Cone(g) \onto{+1}.
$$ 
If we assume two of $f$, $g$ and $gf$ are in $w_{\cN}$, 
then their mapping cones are in $\Nbar$.
Since $\Nbar$ is triangulated, 
the mapping cone of the third one is also in $\Nbar$. 
Then the assertion follows from $\NC1$.

\sn 
$\WE3$ 
For the commutative diagram (\ref{diagram WE3}) 
$$
\xymatrix{
x \ar@{>->}[r] \ar[d]_{a} & y \ar@{->>}[r] \ar[d]_{b} & z \ar[d]_{c}\\ 
x' \ar@{>->}[r] & y' \ar@{->>}[r] & z'& ,
}
$$
in $\cE$. 
We have a conflation in $\cE$
$\Cone(a) \mono \Cone(b) \epi \Cone(c).$
If we assume $\Cone(a)$ and $\Cone(c)$ are in $\cN$, 
then by $\NC3$, $\Cone(b)$ is also in $\cN$.

\sn
$\WE7$ 
For the commutative diagram (\ref{diagram WE7}),  
with $b \in w_{\cN}$. 
We have a retraction $\Cone(a) \to \Cone(b)$
with $\Cone(b)\in\cN$. 
Therefore $\Cone(a) \in \cN$ by $\NC7$. 

\sn
(iii)
For any class of weak equivalences $w$ on $\cE$, 
by Lemma \ref{char of we} we have 
$$
  w_{\cE^w}=\{f\in\Mor\cE\ |\ \Cone(f)\in \cE^w\} = w.
$$
For any null class $\cN$ on $\cE$, we have 
$\cE^{w_{\cN}} = \{x\in\Ob\cE\ |\ Tx=\Cone(x\to 0) \in \cN\}=\cN$, 
where the last equality follows from  ({\bf{NC 2}}).
\end{proof}

Let $\bE = (\cE,w)$ be a complicial exact category with weak equivalences. 
Since the $w$-trivial objects $\cE^w$ is also a complicial in $\cE$, 
we have its stable category $\ul{\cE}^w$ 
which is a full triangulated subcategory of $\ul{\cE}$. 
The two Frobenius categories $\cE$ and $\cE^w$ 
have the same injective-projective objects. 
The inclusion $\cE^w \subset \cE$ induces 
a fully faithful triangulated functor 
$\ul{\cE}^w \to \ul{\cE}$ (\cite{Sch09}, 6.15). 
The triangulated category $\cT(\bE)$ 
associated with $\bE$ 
is the Verdier quotient 
$\cT(\bE) = \ul{\cE}/\ul{\cE}^w$. 
The distinguished triangles in $\cT(\bE)$ 
are triangles which are isomorphic to triangles of the form (\ref{eq:dt}).
The canonical projection  
$\pi:\cE \to \cT(\bE)$ 
from the Frobenius category $\cE$ 
induces an isomorphism 
$w^{-1} \cE \isomto \cT(\bE)$, 
where  $w^{-1}\cE$  is obtained from $\cE$  
by formally inverting the weak equivalences (\cf \cite{Sch09}, Def.\ 6.17). 
Note that $\pi$ sends a (Frobenius) conflation 
to a distinguished triangle in $\cT(\bE)$. 

\begin{lem}[\cite{Sch09}, Exerc.\ 6.16,\ 6.18]
\label{lem:exercise}
$\mathrm{(i)}$ $\Ewbar$ is closed under retract in $\Ebar$. 
In particular, objects of $\Ebar$ which are isomorphic to 
object of $\Ewbar$ in $\Ebar$ are already 
in $\Ewbar$. 

\sn
$\mathrm{(ii)}$ A morphism $f:x\to y$ in $\cE$ is a weak equivalence 
if and only if $\pi(f)$ is an isomorphism in $\cT(\bE)$.
\end{lem}
\begin{proof}
(i) 
Let $a$ be an object in $\Ewbar$ and 
$x$ an object in $\Ebar$. 
Assume that 
there are morphisms $p:a \to x$ and $i:x \to a$ 
such that $pi=\id_x$ in $\Ebar$. 
Hence we have a morphism $H:Cx \to x$ such that 
$H\iota_x=pi-\id_x$ in $\cE$, 
where $\iota_x:x \mono Cx$ is the inflation. 
Then there are morphisms 
$$ 
x \onto{
\left(\begin{smallmatrix}
i\\
-\iota_x
\end{smallmatrix}\right)
} a\oplus Cx \onto{
\left(\begin{smallmatrix}
p & H
\end{smallmatrix}\right)
} x$$
such that 
$ 
\left(\begin{smallmatrix}
p & H
\end{smallmatrix}\right)
\left(\begin{smallmatrix}
i\\
-\iota_x
\end{smallmatrix}\right)
=pi-H\iota_x=pi-(pi-\id_x)=\id_x
$. 
Since $\cE^w$ is closed under homotopy equivalences, 
we have $a\oplus Cx$ is in $\cE^w$ 
and therefore 
$x$ is in $\cE^w$.

\sn
(ii)
By Lemma \ref{char of we}, $f$ is in $w$ if and only if 
$\Cone(f)$ is in $\cE^w$. 
Since $\cE^w$ is closed under homotopy equivalences, 
this condition is equivalent to that 
in $\Ebar$,  
$\Cone(f)$ is in $\Ewbar$. 
By (i), 
$\Ewbar$ is thick in $\Ebar$.
Therefore this condition is equivalent to 
that $\pi(\Cone(f))$ is isomorphic to $0$ in $\cT(\bE)$. 
Since there is a distinguished triangle 
$x \onto{\pi(f)} y \to \Cone(f) \onto{+1}$
in $\cT(\bE)$, the final condition is equivalent to 
that $\pi(f)$ is an isomorphism in $\cT(\bE)$.
\end{proof}

If a full subcategory $\cN$ in $\cE$ 
satisfies $\NC{2}$ and $\NC{3}$, 
$\cN$ is also complicial,  
and the image of $\cN$ 
in the derived category $\TE$  
by the canonical projection $\pi:\E \to \TE$ 
is a triangulated subcategory.  
We define the {\it $w$-closure} $\cN_w$ of $\cN$ by 
the kernel of the composition $\cE \onto{\pi} \TE \to \TE/\pi(\cN)$; 
the full subcategory of $\cE$ 
whose objects are isomorphic to $0$ in $\TE/\pi(\cN)$.

\section{Quasi-weak equivalences}
\label{sec:deloop}
Let $\bE = (\cE,w)$ 
be an {\it idempotent complete} 
exact category with weak equivalences. 
We shall define a class of weak equivalences $qw$ in 
the category of chain complexes $\ChshE$ 
($\# \in \{b,+,-, \emptyset\}$) as follows: 
The quasi-isomorphisms closure of 
$\Ch^b(\cE^w)$ in $\ChshE$ is denoted by $\Ac_w^{\#}(\cE)$ 
which is a null class in $\ChshE$. 
The objects in $\Ac_w^{\#}(\cE)$  are called {\it $w$-acyclic complexes}. 
The class of weak equivalences 
associated with the null class 
$\Ac_w^{\#}(\cE)$ is denoted by $qw$ (Lem.\ \ref{fund bij})
which is called {\it the class of quasi-weak equivalences}. 
The category of chain complexes with quasi-weak equivalences 
$\Ch^{\#}(\bE)=(\ChshE, qw)$ forms 
a complicial exact category with weak equivalences. 
The null class $\Ac_w^{\#}(\cE)$ is the $\qis$-closure of $\Ch^{b}(\cE^w)$.
Hence $qw$ contains $\qis$ 
the class of quasi-isomorphisms in $\Ch^{\#}(\cE)$.

\begin{lem}
\label{lem:w preserve}
Assume that $\cE$ is complicial. 

\sn
$\mathrm{(i)}$ The inclusion functor
$i:\cE \to \ChshE$ sends 
a weak equivalence to a quasi-weak equivalence.

\sn
$\mathrm{(ii)}$ The class $qw$ contains 
chain maps $f:x \to y$ in $\Ch^b(\cE)$
which are degree-wise weak equivalences. 
\end{lem}
\begin{proof}
(i) 
Let $f:x \to y$ be a weak-equivalence in $w$. 
Note that the mapping cone $\Cone(f)$ in $\cE$ 
is in $\cE^w$ (Lem.\ \ref{char of we}). 
Now we have a biCartesian square 
\begin{equation}
\label{diag:w to qw}
\vcenter{
\xymatrix{
x \ar[r]^{a} \ar[d]_{f} & Cx \ar[d]^p  \\
y \ar[r]^-b & \Cone(f) &. 
}}
\end{equation}
The mapping cone of $i(f)$ in $\ChshE$ is 
$\Cone(i(f)) = \cdots \to 0 \to x \onto{f} y \to 0 \to \cdots$.
On the other hand, we have a complex 
$z :=  \cdots \to 0 \to Cx \onto{p} \Cone f \to 0 \to \cdots$ 
in $\Ch^b(\cE^w)$. 
The above diagram (\ref{diag:w to qw})
gives a chain map $\phi:\Cone(i(f)) \to z$. 
Its mapping cone in $\ChshE$ is an acyclic complex
$$
 \cdots \to 0 \to x\to \Cyl(f) \to \Cone(f) \to 0 \to \cdots.
$$
Hence the chain map $\phi$ is a quasi-isomorphism 
and thus we have 
$\Cone(i(f)) \in \Ch^b(\cE^w)_{\qis} = \Chsh(\cE)^{qw}$.

\sn
(ii) 
We show the assertion by 
induction on the length of the complexes $x$ and $y$. 
The brutal truncation $\sigma^{\ge k}x$ is  
$\sigma^{\ge k}x :=  \cdots \to 0 \to 0 \to x^k \to x^{k+1} \to \cdots$ 
and put $\sigma^{<k}x := x/\sigma^{\ge k}x$. 
Consider the following commutative diagram:
$$
  \xymatrix{
    \sigma^{\ge k}x \ar@{>->}[r]\ar[d]^{\sigma^{\ge k}f} & x \ar[d]^f \ar@{->>}[r] & \sigma^{<k}x\ar[d]^{\sigma^{< k}f}\\
    \sigma^{\ge k}y \ar@{>->}[r] & y \ar@{->>}[r] & \sigma^{<k}y&.\\
  }
$$
Here, the horizontal sequences are conflations. 
By the assumption on the induction, we have $\sigma^{\ge k}f$ and $\sigma^{< k}f$ are in $qw$. 
Thus the assertion follows from $\WE3$. 
\end{proof}

We have a triangulated category 
$\cD^{\#}(\cE)  = \cT(\ChshE, \qis)$ 
for $\cE$, 
($\# \in \{b, +, -, \emptyset \})$
that is the derived category 
associated with the complicial exact category 
$\ChshE$ with quasi-isomorphisms as its weak equivalences. 
Similarly, 
when $\bE = (\cE,w)$ is complicial, 
we denote by $\cD^{\#}(\bE) = \cT(\Ch^{\#}(\cE),qw)$   
the {\it derived category} of $\bE$, 
which is defined by the derived category of 
the complicial exact category with weak equivalence 
$\Ch^{\#}(\bE) := (\Ch^{\#}(\cE),qw)$. 
The non-connective $K$-theory spectrum $\bK(\bE)$ 
is defined in \cite{Sch09}.  
Recall that the associated $K$-groups in positive degree 	
are agree with Waldhausen $K$-groups $K_i(\bE)$ 
and $\bK_0(\bE) = K_0(\cT(\bE)^{\sim})$, 
where $\cT(\bE)^{\sim}$  is the idempotent completion of $\cT(\bE)$. 
Note also $\cT(\bE)^{\sim}$ 
becomes a triangulated category (\cite{BS01}). 

\begin{thm}
{\label{partial GW}}
Assume that $\bE = (\cE, w)$ is complicial. 
Then we have the following isomorphisms for each $i \in \Z$:

\sn
$\mathrm{(i)}$ $\bK_i(\bE) \isomto \bK_i(\Ch^b(\bE))$. 

\sn
$\mathrm{(ii)}$ $\bK_i(\Ch^-(\bE)) = \bK_i(\Ch^+(\bE)) = 0$,

\sn
$\mathrm{(iii)}$ $\bK_{i}(\bE) \isomto \bK_{i+1}(\Ch(\bE))$.
\end{thm}
  
\begin{proof}
(i) 
The complicial exact categories $\bE$  
and  $\Chb(\bE)$ admit ({\bf WE 6}).  
Hence they satisfy the {\it factorization axiom}, 
namely, any morphism is a composition of a cofibration followed by 
a weak equivalence (\cite{Sch06}, Appendix). 
By the very definition of the quasi-weak equivalences, 
we have $\Ac_w^b(\cE) = \ChbE^{qw}$ the $w$-acyclic objects in 
$\ChbE$ as null classes. 
Hence we have the following diagram: 
$$
\xymatrix{
\bK(\cE^w) \ar[r]\ar[d]^f & \bK(\cE) \ar[r]\ar[d]^g & \bK(\bE) \ar[d]^h \\
\bK(\Ac_w^b(\cE), \qis) \ar[r] & \bK(\Ch^b(\cE), \qis) \ar[r] & \bK(\Ch^b(\bE)).\\
}
$$
Here, the horizontal sequences are homotopy fibration of spectra 
by the fibration theorem (\cite{Sch06}, Thm.\ 11) 
and the vertical map $h$  is induced from the inclusion map $\cE\to \ChbE$ 
(Lem.\ \ref{lem:w preserve}). 
The vertical map  $g$  is a homotopy equivalence by 
the Gillet-Waldhausen theorem (\cite{Sch06}). 
Hence, it is enough to show the map $f$ is a homotopy equivalence. 
We have $\bK(\cE^w) \isomto \bK(\Ch^b(\cE^w), \qis)$  
from the Gillet-Waldhausen theorem again. 
Recall that $\Ac_w^b(\cE) =\Ch^b(\cE^w)_{\qis}$ is the $\qis$-closure of 
$\Ch^b(\cE^w)$. 
We have 
$\cD^b(\cE)/\cD^b(\cE^w) \isomto \cD^b(\cE)/\cT(\Ch^b(\cE^w)_{\qis},\qis)$.
By comparing the localization sequences, 
$\bK(\Ch^b(\cE^w), \qis) \isomto \bK(\Ac_w^b(\cE), \qis)$ 
and the assertion follows from it.

\sn
(ii) 
Let us consider the endofunctor 
$F={\bigoplus}_{n\in\mathbb{N}} [2n]$ 
on $\Ch^{+}(\cE)$. 
From the identities $F[2]\oplus \id \isomto F$ and 
$\bK(F[2])=\bK(F)$, 
we notice that $\bK(\id)=0$ by the additivity theorem.

\sn 
(iii)
The truncation functor is defined by 
$\tau^{\ge k}x := \cdots \to 0 \to \Im(\partial^{k-1}) \to x^k \to x^{k+1} \to \cdots$, 
for any $x \in \Ch(\cE)$, 
where $\partial^{k-1}:x^{k-1} \to x^{k}$ is 
the differential map. 
The kernel of the quotient map $x \to \tau^{\ge k}x$ 
is denoted by $\tau^{< k}x$. 
For any $x \in \ChbE \cap \Ch^+(\cE)^{qw}$, 
we have quasi-isomorphisms $x\onot {\sim} y \onto{\sim} z$ 
with $y\in \Ch^+(\cE)$ and $z\in \Chb(\cE^w)$. 
Since $\cE$ is idempotent complete, 
the canonical map $\tau^{<k}y \to y$ becomes a quasi-isomorphism 
for some $k$ (\cite{BS01}, Lem.\ 2.6). 
Hence, we have $\ChbE^{qw} = \ChbE \cap \Ch^+(\cE)^{qw}$.
On the other hand, 
for any $x \in \ChbE \cap \Ch^-(\cE)^{qw}$, 
we have quasi-isomorphisms $x\onto {\sim} y \onot{\sim} z$ 
with $y\in \Ch^-(\cE)$ and $z\in \Chb(\cE^w)$. 
The canonical map $y \to \tau^{\ge k}y$ becomes a quasi-isomorphism 
and thus $\ChbE^{qw} = \ChbE \cap \Ch^-(\cE)^{qw}$. 
Similarly, we have $\Ch^+(\cE)^{qw} = \Ch^+(\cE) \cap \ChE^{qw}$, and 
$\Ch^-(\cE)^{qw} = \Ch^-(\cE) \cap \ChE^{qw}$. 
We have the square of fully faithful inclusions which induces on 
category equivalences on the quotient (\cite{Sch06}, Proof of Lem.\ 7):
$$
\xymatrix{
&\cD^b(\cE) \ar[r] \ar[d] &  \cD^{+}(\cE) \ar[d]\\
&\cD^{-}(\cE) \ar[r] & \cD(\cE)&.
}
$$
The diagram extends to 
\begin{equation}
  \label{diag:derived cats}
   \vcenter{\xymatrix{
  & \cD^b(\bE) \ar[r] \ar[d] & \cD^+(\bE)\ar[d] \\
  & \cD^-(\bE) \ar[r]        & \cD(\bE) &.
    }}
\end{equation}
Here, all functors are fully faithful (Lem.\ \ref{lem:tri ker} below) 
and the induced functors on quotients are equivalences. 
Thus the assertion follows from the localization theorem 
and (i) and (ii). 
\end{proof}

\begin{lem}[\cite{Kel96}, Lem.\ 10.3]
\label{lem:tri ker}
Let 
$\cT$ be a triangulated category and 
$\cS$ and $\cN$ full triangulated subcategories of $\cT$. 
If each morphism $x \to y$ in $\cT$ 
with $x \in \cN$ and $y\in \cS$ 
factors through some object in $\cS\cap \cN$, 
then the canonical functor $\cS/\cS\cap\cN \to \cT/\cN$ is fully faithful. 
\end{lem}

Let us denote $n$-th times iteration of $\Ch^{\#}$ for $\bE$ 
by $\Ch^{\#}_n(\bE)$ and 
$\cD^{\#}_n(\bE):=\cT(\Ch_n^{\#}(\bE))$ the {\it{$n$-th higher derived category}} of $\bE$. 
As the following corollary, 
we can consider the negative $K$-groups as obstruction group of idempotent completeness of the higher derived categories:

\begin{cor}
\label{cor:obst}
We assume that 
$\cE$ is complicial or $w$ is just the class of isomorphisms in $\cE$. 
For any positive integer $n$, we have  

\sn
$\mathrm{(i)}$ $\bK_{-n}(\bE) \simeq \bK_0(\cD_n(\bE))$. 

\sn
$\mathrm{(ii)}$ 
$\bK_{-n}(\bE)$ is trivial if and only if 
$\cD_n(\bE)$ is idempotent complete.
\end{cor}
\begin{proof}
Suppose that $\bE = (\cE,w)$ is complicial. 
By Theorem \ref{partial GW} (iii), 
we have 
$\bK_{-n}(\bE)\simeq \bK_0(\Ch_n(\bE))\simeq \bK_0(\cD_n(\bE))$.
Then Proposition \ref{cor:vanish} below leads the desired assertion.  
If $\bE = \cE$ is an exact category; it may not be complical, but $w=$ 
the class of isomorphisms). 
From the Gillet-Waldhausen and Lemma 7 and Corollary 6 in \cite{Sch06} as already refered in Introduction, we have 
$\bK_{-n}(\cE) \isomto \bK_{-n}(\Chb(\cE),\qis) \isomto \bK_{-n+1}(\Ch(\cE),\qis)$. 
Hence, one reduce to the case of complicial.  
\end{proof}

\begin{prop}
  \label{cor:vanish} 
$\mathrm{(i)}$ For an essentially small triangulated category $\cT$, 
if $\bK_0(\cT)=K_0(\cT^{\sim})$ is trivial, 
then $\cT$ is idempotent complete.

\sn
$\mathrm{(ii)}$ 
  The derived category $\cD(\bE)$ is idempotent complete if and only if 
  the Grothendieck group $\bK_0(\cD(\bE))=K_0(\cD(\bE)^{\sim})$ is trivial. 
\end{prop}
\begin{proof}
(i) Since the map $K_0(\cT) \to K_0(\cT^{\sim})$ is injective 
by \cite{Tho97} Corollary\ 2.3, 
now $K_0(\cT)$ is also trivial. 
Applying the Thomason classification theorem of 
  (strictly) dense triangulated subcategories in 
  essentially small triangulated categories \cite{Tho97} Theorem\ 2.1 
  for $\cT^{\sim}$, 
  the inclusion functor $\cT \to \cT^{\sim}$ 
  must be an equivalence.  

\sn
(ii)
  From the diagram (\ref{diag:derived cats}) 
  in the proof of Theorem \ref{partial GW}, 
  we have a surjection 
  $0= K_0(\cD^{+}(\bE)) \oplus K_0(\cD^-(\bE)) \to K_0(\cD(\bE))$. 
  Therefore $K_0(\cD(\bE))=0$. 
  If $\cD(\bE)$ is idempotent complete, that is, 
  $\cD(\bE)\isomto\cD(\bE)^{\sim}$, then we have 
  $\bK_0(\cD(\bE))=K_0(\cD(\bE)^{\sim})=K_0(\cD(\bE))=0$. 
  The converse is followed from (i).
\end{proof}

\providecommand{\bysame}{\leavevmode\hbox to3em{\hrulefill}\thinspace}
\providecommand{\href}[2]{#2}

\noindent
Toshiro Hiranouchi \\
Email address: {\tt hira@hiroshima-u.ac.jp}

\vspace{0.5cm}

\noindent
 Satoshi Mochizuki \\
 {\tt mochi81@hotmail.com}
\end{document}